\documentclass{amsart}
\usepackage{amsmath,amssymb, amsthm, tikz,mathtools, enumerate, graphicx, wasysym}
\usepackage[alphabetic]{amsrefs}
\usepackage[margin=1.25in]{geometry}
\usetikzlibrary{cd}

\usepackage{xcolor}

\newtheorem*{thm*}{Theorem}
\newtheorem{thm}{Theorem}[section]
\newtheorem{lem}[thm]{Lemma}
\newtheorem{cor}[thm]{Corollary}
\newtheorem{prop}[thm]{Proposition}
\newtheorem{conj}[thm]{Conjecture}
\theoremstyle{definition} 
\newtheorem{definition}[thm]{Definition}
\newtheorem{ex}[thm]{Example}

\newtheorem{remark}[thm]{Remark}
\newtheorem{remarks}[thm]{Remarks}

\numberwithin{equation}{section}

\newcommand{\cs}{\ensuremath{C^*}}
\newcommand{\adj}{\ensuremath{\mathcal{L}}}
\newcommand{\bK}{\mathbb{K}}

\newcommand{\Q}{\mathcal{Q}}
\newcommand{\R}{\mathcal{R}}
\newcommand{\id}{\text{id}}
\newcommand{\pr}{\text{pr}}
\newcommand{\sk}{\times_\kappa}
\newcommand{\supp}{\text{supp}}
\newcommand{\rg}{\text{rg}}
\newcommand{\inv}{^{-1}}
\newcommand{\rt}{\text{rt}}
 
\newcommand{\proj}{\text{proj}}

\title{Coactions and Skew Products for Toplogical Quivers}

\author[Hall]{Lucas Hall}
\address{School of Mathematical and Statistical Sciences
\\Arizona State University
\\Tempe, Arizona 85287}
\email{lhall10@asu.edu}

\begin{document}
\maketitle

\begin{abstract}
Given a cocycle on a topological quiver by a locally compact group, the author constructs a skew product topological quiver, and determines conditions under which a topological quiver can be identified as a skew product. We investigate the relationship between the \cs-algebra of the skew product and a certain native coaction on the \cs-algebra of the original quiver, finding that the crossed product by the coaction is isomorphic to the skew product. As an application, we show that the reduced crossed product by the dual action is Morita equivalent to the \cs-algebra of the original quiver.
\end{abstract}

\section{Introduction}

The dynamics of \cs-algebras has long been an area of interest for operator algebraists. This theory of crossed product duality simplifies when acting groups are abelian, but the wide array of uses left a desire for generalizations to duality suitable for nonabelian groups. This desire was furnished by the duality theorem of Imai and Takai, which depended on the design of coactions to replace the role of the dual group in the nonabelian setting. While these objects greatly generalized dynamical results, the cost was significant. The loss of the dual group in favor of coactions in this setting cast this new field into a challenging intersection of tensor products and multiplier algebras. 

Independently, the introduction of graph algebras in the late twentieth century shed much new light in the field of operator algebras. Popularized as \cs-algebras that you can ``see,'' graph algebras succinctly relate the algebraic structure of the \cs-algebra to observable combinatorial data associated to the defining (directed) graph. Their ongoing study is simultaneously tractable and fruitful, since on one hand the hope of using simple pictures to analyze a \cs-algebra is both clear and stimulating, while on the other hand \cs-algebras from across the field have found natural expressions as graph algebras. Once again, this wide array of uses has produced an interest in generalizing the constructions of graph algebras, which have been pursued along many different lines. Examples include higher ranked graphs, Exel-Laca algebras, topological graphs, and the subject of this paper, topological quivers. 

One boon graph algebras offered to the theory of nonabelian duality is the incredible leverage offered by ``seeing'' how the coactions work in the context of graph algebras. The first result of this type \cite{kqr:graph} considered labelings of the edge set by a discrete group, introducing a new directed graph whose vertex and edge sets are products of the originals with the discrete group. At the same time, this labeling induced a coaction on the original graph algebra, and the authors showed in \cite{kqr:graph}*{Theorem 2.4} that these two constructions are isomorphic using the technology devoted to graph algebras, namely Cuntz-Krieger families and the Gauge-Invariant Uniqueness Theorem. 

Later, Katsura codefied the data of a graph algebra into a \cs-correspondence \cite{kat:class}, and offered the modern perspective \cite{kat02} that the graph algebra of a specified graph is the universal \cs-algebra for (what we would call today) Cuntz-Pimsner covariant homomorphisms of the associated \cs-correspondence. Here, universal means that every Cuntz-Pimsner covariant representation of the correspondence into a \cs-algebra induces a homomorphism from the universal algebra into the representing algebra. Simultaneously, Katsura introduced the notion of topological graphs, which offered the first continuous analog to directed graphs by puting topological constraints on the data. 

Equipped with the technology of Cuntz-Pimsner covariant homomorphisms into \cs-algebras, Kaliszewski, Quigg, and Robertson introduced the notion of Cuntz-Pimsner covariance for morphisms \emph{between correspondences} \cite{kqrfunctor}. There, they showed that the passage from \cs-cor\-re\-spon\-dences to the Cuntz-Pimsner algebra was functorial between a suitible pair of categories. Shortly after, these three authors investigated this functor for correspondences equipped with a coaction\cite{kqrcoact}, deducing that a suitable correspondence coaction generates a coaction on the associated Cuntz-Pimsner algebra. Kaliszewski and Quigg then used this machinery to recover for Katsura's topological graphs what had been discovered for directed graphs \cite{tgcoact}. The skew product topological graph by a (locally compact) group is isomorphic to the crossed product by a naturally occuring coaction, producing this time a coaction that you can ``see'' in the continuous setting.  

Around the time that Katsura studied his topological graphs, Muhly and Solel introduced topological quivers to exemplify traits in the study of Morita equivalence of tensor algebras. Muhly and Tomforde later took up the task of investigating the associated \cs-algebras in \cite{quiver}. Apart from the matter of convention (and we conform to the conventions of Katsura here), the principle difference between quivers and topological graphs is that topological graphs insist that the source map is a local homeomorphism, while quivers insist only that the source map is open. To make up for this freedom, we insist that the fibres over any vertex can be appropriately measured, and the authors follow the strategy used for measuring groupoids by a system of measures. In this article, we take up the study of skew products for topological quivers. 

In \S\ref{prelim}, we record our conventions for bundles, topological quivers and associated constructions, and coactions. Starting with \S\ref{skew} we introduce the skew product quiver and produce a Gross-Tucker type classification for skew products. In \S\ref{algebra} we study the associated \cs-algebras in earnest, culminating in the main objective Theorem \ref{isomorphism}. We end with a section on applications, where we offer an open problem regarding qualities of the coaction, and show that the reduced crossed product by the dual action is Morita equivalent to the \cs-algebra of the original quiver.  

\subsection*{Acknowledgements} The author would like to thank John Quigg for numerous discussions and insights regarding this work. Without his patient encouragement and diligent guidance, this work would be far inferior. 

\section{Preliminaries}\label{prelim}

\subsection{Bundle Constructions}

In the development of bundles, we follow \cite{husemoller}, but we define principal bundles accoring to the hypothesis of \cite{RaeWil}*{Lemma 4.63}. The remark following the lemma reconciles the disparity in definitions.  


\begin{definition}
A \emph{bundle} is a triple $(E, p, B)$, where $p:E\to B$ is a continuous function. $B$ is called the base space, $E$ the total space, $p$ the bundle projection, and for each $b\in B$, the space $p^{-1}(b)$ is called the fibre over $b$. We often identify a bundle with its bundle projection. 

A bundle morphism is a pair of continuous maps $(u, f): (E, p, B)\to (E', p', B')$ with $p'u=fp$, i.e. the diagram below commutes. 
\begin{equation*}
\begin{tikzcd}
E\ar[swap, "p"]{d}\ar["u"]{r} & E'\ar["p'"]{d}\\
B\ar["f", swap]{r} & B'.\\
\end{tikzcd}
\end{equation*}
An bundle \emph{isomorphism} is an invertible bundle morphism (which requires each map $u, f$ to be a homeomorphism). 
\end{definition}
\begin{ex}
Given a bundle $(E, p, B)$ and a continuous map $f: C \to B$, the space 
\[E*_f C=\{(x, c)\in E\times C | p(x) = f(c)\}\]
 is called the pullback under $f$, and the triple $(E*_f C, \pr_2, C)$ is a bundle, where $\pr_2$ is the standard projection onto the second factor. The pair $(\pr_2, f)$ is then a bundle morphism. 
\end{ex}
\begin{ex}
A topological group $G$ acts on on the right a locally compact Hausdorff space $X$ provided there is a continuous group anti-homomorphism $\alpha:G\to \text{Homeo}(X)$ (written $\alpha(g)(x)=xg$), and we call such a space a right $G$-space. $G$ acts \emph{freely} on $X$ provided all stabilizers are trivial.  $G$ acts \emph{properly} on $X$ if the translation map $(x, g)\mapsto (x, xg)$ is a proper map. The action is called \emph{principal} if $G$ acts both freely and properly. If $G$ acts on the right of $X$, then the triple $(X, q, X/ G)$ is a bundle, where $X/ G$ (the set of orbits of $G$) is given the quotient topology induced by the quotient map $q$. This is called the quotient bundle for the $G$ action. When $G$ acts principally on $X$, then the space $X/ G$ is also locally compact and Hausdorff.  

For $g\in G$, the pair $(\rt_g, \id)$ is a bundle morphism of $q$, where $\rt_g(x) = xg$ is right translation by $g$, and this morphism is invertible.  
\end{ex}
\begin{definition}
Given a locally compact group $G$, a bundle $p$ is a \emph{$G$-bundle} provided the total space $E$ is a $G$-space and there is a bundle isomorphism $(\id,f): (E, p, B)\to (E, q, E\backslash G)$. A $G$-bundle $p$ is called \emph{principal} if $G$ acts freely and properly on the total space. In this case, the fibre $p^{-1}(b)$ over any point is isomorphic to $G$. 

A \emph{$G$-bundle morphism} is a bundle morphism $(u, f):(E, p, B)\to (E', p', B')$ whose map of total spaces is equivariant for the $G$-action, that is $u(xg) = u(x)g$.  
\end{definition}

\begin{ex}
If $(E, p, B)$ is a $G$-bundle and $f:C\to B$ is continuous, then there is a $G$-action on $C*_f E$ given by $(c, e)g = (c, eg)$ making the pullback into a $G$-bundle over $C$. The pullback bundle inherits any structure enjoyed by the original bundle. 
\end{ex}

The following theorem will be foundational for us. For details, see \cite[Theorem 4.4.2]{husemoller} and the discussion surrounding \cite[Proposition 1.3.4]{palais}. This result also holds for groupoids, see \cite{hkqwstab} and the references therein.


 \begin{thm}\label{palais}
Suppose we have principal $G$-bundles $p:E\to B$ and $p':E'\to B'$, together with a $G$-bundle morphism $(u, f):(E, p, B)\to (E', p', B')$. Then the assignment
\begin{align*}
E&\to B*_f E'\\
e&\mapsto (p(e), u(e)),
\end{align*} 
is a $G$-equivariant homeomorphism between $E'$ and $B*_f E'$. In particular, this map produces a $G$-bundle isomorphism from $E$ to $B*_f E'$ over $B$.
 \end{thm}

\begin{definition}[\cite{wil:toolkit}*{\S 3.2}]
Given a bundle $\pi: E\to B$, a system of measures over $\pi$, or \emph{$\pi$-system}, is a family of positive Radon measures $\lambda = \{\lambda_b: b\in B\}$ defined on $E$ and satisfying two conditions
\begin{enumerate}
\item the support of $\lambda_b$ is a subset of the fibre over $b$, and 
\item\label{cont-assembly} for any $f\in C_c(E)$, the assignment \[b\mapsto \lambda(f)(b) := \int f d\lambda_b\] is a continuous function.
\end{enumerate}
Condition 2 is called continuity for the system. A $\pi$-system is called \emph{full} provided $\text{supp}(\lambda_b) = \pi^{-1}(b)$, and \emph{proper} provided $\lambda_b\neq 0$ for every $b\in B$. We will confine our attention to full and proper $\pi$-systems.  
\end{definition}
Suppose the spaces $E, B$ admit (right) actions by a group $G$ and $\pi$ is equivariant for the action. A $\pi$-system is called \emph{equivariant} provided 
\[\int_E f(eg)d\lambda_b(e) = \int_E f(e)d\lambda_{bg}(e)\] 
for every $f\in C_c(E)$, $b\in B$, and $g\in G$. 

If $G$ acts principally on each of $E, B$ and if $\pi$ is equivariant and surjective, then there is a canonical map $\dot{\pi}$ making the diagram 
\begin{equation*}
\begin{tikzcd}
E\ar["q_E", swap]{d}\ar["\pi",]{r} & B\ar["q_B"]{d}\\
E\backslash G\ar["\dot{\pi}", swap]{r} & B\backslash G
\end{tikzcd}
\end{equation*}
commute. The orbit map for $E$ induces a homeomorphism $\pi^{-1}(b)\to \dot{\pi}^{-1}(bG)$, so that equivariance of $\lambda$ induces for each orbit $bG\in B\backslash G$ a measure $\mu_{bG}$ with support $\dot{\pi}(bG)$. For $\phi\in C_c(E\backslash G)$, this measure is given by 
\[\int \phi d\mu_{bG}.\]

In this circumstance, we have the following theorem relating $G$- equivariant $\pi$-systems to those of their orbits.

\begin{thm}[\cite{wil:toolkit}*{Proposition 3.14}]\label{orbitsystems}
Let $G$ be a group, $E$ and  $B$ principal right $G$-spaces, and $\pi: E\to B$ a continuous open surjection. Let $\dot{\pi}$ be the induced map for the orbit spaces of $E$ and $B$. 
\begin{enumerate}[(a)]
\item If $\lambda$ is an invariant $\pi$-system, then the induced family of measures $\mu=\{\mu_{bG}\}_{bG\in B\backslash G}$ is a $\dot{\pi}$-system.
\item If $\mu$ is a $\dot\pi$-system, then there is a $\pi$-system $\lambda$ whose induced $\dot\pi$-system agrees with $\mu$. 
\end{enumerate}
\end{thm} 


\subsection*{ \textbf{\cs}-Constructions for Topological Quivers}

We refer to \cite{quiver} for topological quivers, but adapt the definitions to the conventions of \cite{rae:graphalg, kat:class}, which we use along with \cite{enchilada}.

\begin{definition}
A \textit{Topological Quiver} is an ensemble $\mathcal{Q}=\{E^0,E^1, r, s, \lambda\}$ comprising the following data:
\begin{itemize}
\item $E^1$ and $E^0$ are locally compact Hausdorff spaces.
\item $r$ and $s$ are continuous from $E^1$ to $E^0$.
\item $s$ is open.
\item $\lambda$ is a full $s$-system
\end{itemize}
\end{definition}
\begin{remarks}
$r$ and $s$ are called the range map and source map, elements of $E^1$ are typically called edges, and elements of $E^0$ vertices. These names are anticipated by the analogy of quivers with directed graphs \cite{rae:graphalg}, which are essentially quivers whose edge and vertex spaces are discrete and countable, and the family $\lambda$ (typically suppressed in this context) consists of  counting measures. A vertex which is in the image of the source map but not the range map is called a source, and a vertex which is in the image of the range map but not the source map is called a sink. Using the technique of adding tails to quivers (cf. \cite[Section 4]{quiver}), one can always restrict their attention to quivers without sinks. We will often briefly call a topological quiver a quiver.
\end{remarks}

Given a topological quiver $\mathcal{Q}$, \cite{quiver} defines  a quiver $C^*$-correspondence as follows: Let $A=C_0(E^0)$. For $X_0=C_c(E^1)$, there is an $A$-valued inner product $\langle\cdot,\cdot\rangle:X_0\times X_0\to A$ by \[\langle\eta,\xi\rangle(v) = \int \overline{\eta(e)}\xi(e)\,d\lambda_v(e).\]
We then get left and right actions of $f\in A$ on $\xi \in X_0$ by 
\begin{align*} 
f\xi(e) &= f(r(e))\xi(e) \text{ and}\\
 \xi f(e) &= \xi(e)f(s(e)),
 \end{align*} 
 which encode in the module $X_0$ the structure of the topological quiver. Denoting by $X$ the completion of $X_0$ with respect to this inner product, $X$ is an $A-A$ correspondence. In fact, for $A'=C_0(E^1)$, $X$ is also an $A'-A$ correspondence with the left action given by pointwise multiplication. Notice that, whether we regard $X$ as an $A-A$ correspondence or an $A'-A$ correspondence, $X$ is nondegenerate, by which we mean that $\overline{\text{span}}(A\cdot X) = \overline{\text{span}}(A'\cdot X) = X$. In fact, the Cohen-Hewitt Factorization Theorem guarantees that $A\cdot X= X$, so authors will often take this to be the definition of nondegeneracy for \cs-correspondences. 
  
 \begin{remark}
When we declare a left action of a \cs-algebra $A$ on $X$, what we really mean is a homomorphism $\phi_A: A\to \mathcal{L}(X)$. In particular, that the left action of $A'$ on $X$ is implimented by pointwise multiplication, and the left action by $A$ is induced by the range map by assigning $f\mapsto f\circ r$ and then acting by pointwise multiplication. In our context, this homomorphism $\phi_A$ is faithful when the set of sources has empty interior. 
 \end{remark}

 Given the \cs-correspondence of a topological quiver, we are interested in constructing a \cs-algebra which reflects the quiver's data. We will also find value in representing the correspondence of one quiver in the correspondence of another quiver in a way which preserves the ``essential'' structure of interest to us. In what follows, recall that for an $A-A$ correspondence $X$, the multiplier correspondence is $M(X)=\mathcal{L}_A(A,X),$ the space of all adjointable $A$-linear operators from $A$ (regarded as the standard correspondence over itself) into $X$. $M(X)$ is an $M(A)-M(A)$ correspondence. Given a pair $(X, A)$, $(Y,B)$ of correspondences, a \emph{morphism} of correspondences is a pair $(\psi, \pi):(X, A)\to (M(Y),M(B))$ of maps consisting of a homomorphism $\pi:A\to M(B)$ and a linear map $\psi:X\to M(Y)$ which preserves the correspondence operations. A \emph{representation} of $(X, A)$ in a \cs-algebra $D$ is a correspondence homomorphism $(\psi, \pi):(X,A)\to (M(D), M(D))$ into the multiplier of the standard correspondence associated to $D$. We are interested in a particular class of correspondence homomorphisms.

\begin{definition}[\cite{kqrfunctor}*{Definition 3.1}]\label{CPcovcor}
A correspondence homomorphism is \emph{Cuntz-Pimsner covariant} provided 
\begin{enumerate}[(i)]
\item $\psi(X)\subset M_B(Y)$
\item $\pi:A\to M(B)$ is nondegenerate
\item $\pi(J_X)\subset M(B; J_Y)$
\item The following diagram commutes 
\end{enumerate}
\begin{equation*}
\begin{tikzcd}
J_X\ar["\pi|"]{r}\ar["\phi_A|", swap]{d} & M(B; J_Y)\ar["\overline{\phi_B}|"]{d}\\
\bK(X)\ar["\psi^{(1)}", swap]{r} & M_B(\bK(Y)).\\
\end{tikzcd}
\end{equation*}
\end{definition}

\begin{remark}
Above, $J_X$ is the largest ideal mapping $A$ injectively into $\bK(X)$ (likewise for $J_Y$), $\psi^{(1)}$ is the naturally occurring map from $\bK(X)$ to $M(\bK(Y))$, and 
\[
M(B; J_Y):=\{m\in M(B)|\, mB\cup Bm\subset J_Y\}.
\] 

For topological quivers, \cite{quiver}*{Proposition 3.15} provides a satisfactory description of $J_X$ in terms of the \emph{regular verticies} of the topological quiver, themselves defined in \cite{quiver}*{Definition 3.14}. The regular verticies are characterized as 
\[
E^0_{\text{reg}}=E^0_{\text{fin}}\cap \text{Int}(\overline{r(E^1)}),
\]
where
\begin{align*}
E^0_{\text{fin}} = \{v\in E^0 : \text{there exists a precompact neighborhood $V$ of $v$ such that } r^{-1}(\overline{V}) \\
\text{is compact and } s|_{r^{-1}(V)} \text{ is a local homeomorphism}\}.\
\end{align*}
\end{remark}

According to definition \ref{CPcovcor}, \cite{kqrfunctor}*{Theorem 3.7} defines a functor from the category of \cs-correspondences with Cuntz-Pimsner covariant homomorphisms to the category of \cs-algebras with nondegenerate homomorphisms. Through the remainder of this article, we denote by $(X,A)$ the $A-A$ correspondence associated to $\Q$ and $(\cs(\Q), k_X, k_A)$ the universal Cuntz-Pimsner covariant representation of $(X,A)$. This universal \cs-algebra is the definition of the \cs-algebra associated to $\Q$.

\subsection*{Coaction Constructions}

For a detailed exposition on coactions, one can consult the appendix to \cite{enchilada}. Antecedents for this article are due to \cite{kqrcoact}.

\begin{definition}
Given a locally compact group $G$ and a \cs-algebra $A$, a coaction $(A,\delta)$ of $G$ is an injective nondegenerate homomorphism 
\[\delta:A\to M(A\otimes \cs (G))\] which satisfies the \emph{coassociative condition,}
\[(\delta\otimes\id)\circ\delta = (\id\otimes\delta_G)\circ\delta\] together with \emph{coaction nondegeneracy}
\[\overline{\text{span}}\{\delta(A)\cdot (1\otimes \cs (G))\} = A\otimes\cs (G).\]
\end{definition}

\begin{remark}
 The theory of coactions makes emphatic use of the group \cs-algebra together with tensor products, and these structures are deeply bound to the theory. Consider, for example, that for a locally compact group $G$, we can associate with the canonical homomorphism $\iota: G\hookrightarrow M(\cs (G))$ a distinguished unitary $w_G\in UM(C_0(G)\otimes\cs (G))$ as follows. We identify $M(C_0(G)\otimes \cs (G))$ with $C_b(G, M^\beta(\cs (G) ))$, the bounded continuous functions from $G$ into $M(\cs (G) )$ endowed with the strict topology, and then define $w_G$ by $w_G(g)=\iota(g)\in M(\cs (G) )$. This unitary casts the embedding $\iota$ into the theory of tensor products, which allows us to describe covariant representations of coactions in what follows.

 One special example of a coaction is implimented by a homomorphism $\psi:C_0(G)\to M(A)$, according to the assignment 
 \begin{align*}
 \delta^\psi:A&\to M(A\otimes \cs (G) )\\
  a&\mapsto  \text{Ad}(\psi\otimes\id(w_G))(a\otimes 1)\\
  \end{align*} 
  and such a coaction is termed \emph{inner}. Much as in the theory of actions, we define a covariant representation of $(A, \delta)$ as a pair of nondegenerate homomorphisms $(\pi, \mu)$ such that the diagram
 \begin{equation*}
\begin{tikzcd}
&A\ar["\delta"]{r}\ar["\pi", swap]{d} & M(A\otimes \cs (G) )\ar["\pi\otimes \id"]{d}\\
C_0(G)\ar["\mu",swap]{r}&M(B)\ar["\delta^\mu", swap]{r} & M(B\otimes \cs (G) )\\
\end{tikzcd}
 \end{equation*}
 commutes. Loosely speaking, a covariant representation provides a setting where the coaction $\delta$ as an inner coaction.
\end{remark}

\begin{remark}
In the full generality of locally compact groups, it is necessary for us to work with multiplier algebras frequently. The use of multiplier algebras affords us the power to consider strict extentions of homomorphisms, which under nondegeneracy conditions are unique. There are natural questions about the compatibility of strict extentions under composition, which in practice cause us no concern. Consequently, by abuse of notation we commonly will identify a homomorphism with its strict extension. One can follow \cite{boiler}*{appendix} for more details on the general theory.  
\end{remark}

\begin{definition}
A crossed product for a coaction is a \cs-algebra $A\rtimes_\delta G$ together with a covariant representation $(j_A, j_G): (A,\delta)\to M(A\rtimes_\delta G)$ that is universal for covariant representations. That is, if $(\pi, \mu): (A,\delta)\to M(B)$ is any covariant representation of the coaction, there is a unique nondegenerate homomorphism $\pi\times\mu: A\rtimes_\delta G\to M(B)$ with the property that 
\begin{align*}
\pi=(\pi\times\mu)\circ j_A\\
\mu=(\pi\times\mu)\circ j_G.\\
\end{align*}
 $\pi\times \mu$ is called the integrated form of the representation $(\pi, \mu)$. 
\end{definition}

\begin{remarks}
Let $(A, \delta, G)$ be a coaction, and consider the multiplication representation $M:C_0(G)\to M(\bK(L^2 (G) ))$. Then one covariant representation for $(A,\delta)$ is implimeneted by the pair 
\[
((\id\otimes\lambda)\circ\delta, 1\otimes M):(A,\delta)\to M(A\otimes \bK(L^2(G)))
\]
which is called the \emph{regular} representation of the coaction. The integrated form of this representation is always faithful, indicating a departure from the theory of actions on \cs-algebras. Instead, the notion of ``intermediate'' completions for \cs-dynamical systems is recovered in this dual category by different classes of coactions. We will shortly introduce the two most important classes.

Fix $s\in G$. Another example of a representation for $(A,\delta)$ is given by the pair $(j_A, j_G\circ \rt_s):(A,\delta)\to M(A\times_\delta G)$, which is an automorphism of $A\times_\delta G$. The family of integrated forms indexed by the elements of $G$ produces an action of $G$ on the crossed product, called the \emph{dual} action to $\delta$, and denoted by $\hat\delta$. 
\end{remarks}

\begin{definition}[cf. \cite{boiler}, \cite{maximal}]

A coaction $(A,\delta)$ is called \emph{normal} provided the canonical homomorphism $j_A:A\to A\times_\delta G$ is faithful. 

A coaction $(A, \delta)$ is called \emph{maximal} provided there is an isomorphism 
\[
A\times_\delta G\times_{\hat{\delta}} G \cong A\otimes \bK(L^2(G))
\] 
given by a certain canonical map. 
\end{definition} 

A given coaction $(A,\delta)$ always has both a \emph{normalization} $(A^n,\delta^n)$ and a \emph{maximalization} $(A^m, \delta^m)$, and there are natural surjections $A\to A^n$ and $A^m\to A$. It is possible for a coaction to be both maximal and normal. 
This has strong implications for the dual action, as we see below. 

\begin{prop}\label{amenableaction}
Suppose $(A, \delta)$ is any coaction which is both normal and maximal. Then the dual action $\hat{\delta}$ is amenable in the sense that the regular representation
\[A\times_\delta G\times_{\hat{\delta}} G \to A\times_\delta G\times_{\hat{\delta}, r} G\]
is an isomorphism.
\end{prop}

\begin{proof}

This is an immediate consequence of \cite{maximal}*{Proposition 2.2}, which the authors illustrate pictorally  with the commutative diagram 
\begin{center}
\begin{tikzcd}
A\times_\delta G\times_{\hat{\delta}} G\ar["\Phi"]{r} \ar["\Lambda", swap]{d} & A\otimes \bK(L^2(G))\ar["\psi\otimes\id"]{d}\ar["\Psi"]{dl}\\
A\times_\delta G\times_{\hat{\delta}, r} G\ar["\Upsilon", swap]{r} & A^n\otimes \bK(L^2(G)).\\
\end{tikzcd}
\end{center}

In the diagram, $\Lambda$ is the regular representation for the crossed product, $\Phi$ is the canonical surjection, and $\psi$ is the quotient map $\psi:A\to A/\ker{j_A}$ associated to the universal map $j_A:A\to A\times_\delta G$. Their proof furnishes the isomorphism $\Upsilon$ and the surjection $\Psi$. 

Since $\delta$ is a maximal coaction, the map $\Phi$ is injective, and since $\delta$ is normal, $\ker{j_A}=\{0\}$ and $\psi$ is the identity. It now follows that 
\[
\Lambda = \Psi\circ\Phi = \Upsilon^{-1}\circ (\psi\otimes\id)\circ\Phi,
\]
a composition of isomorphisms.
\end{proof}

Finally, we will make use of the notion of coactions for \cs-correspondences. The relevant definition is supplied below. 

\begin{definition}
Let $X$ be an $A-B$ correspondence, $G$ a locally compact group, and $\delta, \epsilon$ coactions of $G$ on $A$ and $B$ respectively. A coaction of $G$ on $X$ is a nondegenerate correspondence homomorphism \[\sigma:X\to M(X\otimes C^*(G))\] which has $\delta$ and $\epsilon$ as coefficient maps, and such that 
\begin{enumerate}[(i)]
	\item $\overline{\text{span}}((1_{M(A)}\otimes C^*(G))\sigma(X)) = X\otimes C^*(G)$
	\item $(\sigma\otimes\text{id}_G)\circ \sigma = (\text{id}_X\otimes \delta_G)\circ \sigma$.
\end{enumerate} 
\end{definition}

\section{Skew Product Quivers}\label{skew}

\subsection{Definition of the Skew Product Quiver}
We wish to define a skew product quiver in analogy with the skew product graphs and skew product topological graphs studied to date \cite{ kqr:graph, tgcoact}. So let $G$ be a locally compact group, and let $\kappa:E^1\to G$ be a continuous map, called a \textit{cocycle} of $G$ on $\mathcal{Q}$. Then if $F^i=E^i\times G$ with continuous range map $r_\kappa(e,g)=(r(e),\kappa(e)g)$ and continuous, source map $s_\kappa(e, g)=(s(e),g)$, we have most of the data required for a topological quiver. All that remains is to find a family of measures on $F^1$ satisfying our conditions above. Since $s^{-1}(v,g)= s^{-1}(v)\times \{g\}$, we define $\lambda_\kappa=\{\lambda_{v,g}\}_{(v,g)\in F^0}$ by $\lambda_{v,g}(E)=\lambda_v(\{e\in E^1| (e,g)\in E\})$. First, we handle a technical lemma. 

\begin{lem}\label{uniform}
Let $\xi\in C_c(E^1\times G)$, and $(g_i)_{i\in I}$ a convergent net in $G$ with limit $g$. Then the net of functions $(\xi(\cdot, g_i))_{i\in I}$ converges uniformly to $\xi(\cdot, g)$. 
\end{lem}

\begin{proof}
Certainly, $\xi(\cdot, g_i)$ converges pointwise to $\xi(\cdot, g)$ by the continuity of $\xi$. The claim is that this convergence is uniform. 

Consider the space $K\times L\subset E^1\times G$ defined by $K=\proj_{E^1}(\supp(\xi))$ and $L=\proj_G(\supp(\xi))$, which is a compact subset of the product. Choosing some precompact open set $W$ in $E^1\times G$ containing the slice $K\times \{g\}$, we apply the Tube Lemma \cite{munkres}*{Lemma 26.8} to $W\cap K\times L$ find a product $K\times V$ of open sets with $K\times \{g\}\subset K\times V\subset W$. Writing $U=\proj_{E^1}(W)$, 
the product $U\times V$ is a precompact open set containing $\supp(\xi(\cdot, g))$. 

We now claim that the family of functions $\{\xi(\cdot, g)\}_{g\in \overline{V}}$ satisfies the hypotheses of Ascoli's Theorem. That the family is pointwise bounded is clear, and equicontinuity follows from a standard $\epsilon/3$ argument. So the family of functions is precompact in $C_0(E^1)$. Since the net $(\xi(\cdot, g_i))$ is eventually in the family, there is a convergent subnet for the uniform topology. Since the net converges pointwise to $\xi(\cdot, g)$, the net must converge uniformly to the same limit. 
\end{proof}

\begin{prop}
The ensemble $\mathcal{Q}\times_\kappa G=\{F^0, F^1, r_\kappa, s_\kappa, \lambda_\kappa \}$ defines a topological quiver.
\end{prop}

\begin{proof}
Clearly $r_\kappa$ and $s_\kappa$ are continuous, with $s_\kappa$ an open map. Since 
\[
s_\kappa^{-1}(v, g) = \{(e,g)|\,s(e)=v \} =  s^{-1}(v)\times \{g\},
\] we have $\text{supp}(\lambda_{v, g}) = s^{-1}(v,g)$. So all that remains is to show that the family $\lambda_\kappa$ is a continuous $s_\kappa$-system.    

Let $\xi\in C_c(F^1)$, and take a net $(v,g)_i=(v_i,g_i)$ converging to $(v,g)$ in $F^0$. Then compute  
\begin{align*}
&\left|\int\xi\,d\lambda_{v,g}-\int\xi\,d\lambda_{v_i,g_i }\right|=\left|\int \xi(\cdot, g)\,d\lambda_{v}- \int \xi(\cdot, g_i)\,d\lambda_{v_i}\right|\\
\leq& \left|\int \xi(\cdot, g)\,d\lambda_v - \int \xi(\cdot, g)\,d\lambda_{v_i}\right|+\left|\int \xi(\cdot, g)\,d\lambda_{v_i} - \int \xi(\cdot, g_i)\,d\lambda_{v_i}\right|\\
=& \left|\int \xi(\cdot, g)\,d\lambda_v - \int \xi(\cdot, g)\,d\lambda_{v_i}\right|+ \left|\int \xi(\cdot, g)-\xi(\cdot, g_i)\,d\lambda_{v_i}\right|\\
\leq& \left|\int \xi(\cdot, g)\,d\lambda_v - \int \xi(\cdot, g)\,d\lambda_{v_i}\right|+ \int\left| \xi(\cdot, g)-\xi(\cdot, g_i)\right|\,d\lambda_{v_i}
\end{align*}

We begin by analysing the second term above. Let $\epsilon>0$.  By Lemma \ref{uniform}, there is an index $i$ such that $j\geq i$ implies $\|\xi(\cdot, g)-\xi(\cdot, g_j)\|\leq \epsilon$. Meanwhile, by Urysohn's Lemma we may find a compactly supported function $\eta$ with $\eta\equiv 1$ on $\proj_{E^1}(\supp(\xi))$. It now follows for the second term that sufficiently far along the sequence
\begin{align*}
\int\left|\xi(\cdot, g)-\xi(\cdot, g_i)\right|\,d\lambda_{v_i}
&\leq \int\epsilon \eta\,d\lambda_{v_i} =\epsilon \int \eta\,d\lambda_{v_i}
\end{align*}
which converges by continuity to $\epsilon \int\eta\,d\lambda_v.$ Now, 
\begin{align*}&\left|\int \xi(\cdot, g)\,d\lambda_v - \int \xi(\cdot, g)\,d\lambda_{v_i}\right|+ \int\left| \xi(\cdot, g)-\xi(\cdot, g_i)\right|\,d\lambda_{v_i}\\
\leq & \left|\int \xi(\cdot, g)\,d\lambda_v - \int \xi(\cdot, g)\,d\lambda_{v_i}\right|+ \epsilon\int\eta\,d\lambda_{v_i}\\
= & \left|\int \xi(\cdot, g)\,d\lambda_v - \int \xi(\cdot, g)\,d\lambda_{v_i}\right|+ \epsilon\left( \int\eta\,d\lambda_{v_i} - \int\eta\,d\lambda_v\right) + \epsilon\left(\int\eta\,d\lambda_v\right).
\end{align*}

Now, each of the three terms vanish along the sequence by either continuity of the $s$-system or Lemma \ref{uniform}, establishing continuity for $\lambda_\kappa$.


\end{proof}

\begin{remark}\label{qaction}
Now that we have seen that $\Q\times_\kappa G$ is a topological quiver, observe how $G$ interacts with the skew product. The edge and vertex sets $E^i\times G$ are trivial bundles over $G$, and we can associate to each individual space $E^i\times G$ an action of $G$ on the right by $(e, h)\cdot g=(e, hg)$. Since $G$ acts on itself and only on the second coordinate, this defines a principal $G$-bundle $q^i:F^i\to E^i$. In practice, we will freely write $q$ to refer to either of these quotient maps. Moreover, the range and source maps each define $G$-bundle morphisms from the edge spaces to the vertex spaces. 

Since we have these quotient maps between topological spaces, we also have an induced nondegenerate $*$-homomorphism $C_0(E^0)\to M(C_0(F^0))$ by $f\mapsto f\circ q$. 
\end{remark}

\subsection{A Classification of Skew Products} 

In this section we extend the notions of \cite[section 3]{dkq} with the aim of showing that the only obstructions to a topological quiver being a skew product are the existence of a free and proper action by $G$ and a global section for the quotient. We remark that \cite{dkq} goes considerably further in their analysis of group actions on topological graphs, which we leave to pursue in future work. 

Given a topological quiver $\Q$, a locally compact group $G$, and a cocycle $\kappa:E^1\to G$, we have shown how to construct a skew product quiver. One lifts the associated structures (being the range map, source map, and the $s$-system $\lambda$) to the appropriate trivial bundles $E^i\times G$, allowing some interference by the cocycle. Investigating this construction further, we have seen that each of the trivial bundles is equipped with a free and proper action of $G$ by right translation in the second variable, and that the quotient of each space by this action recovers the original edge and vertex spaces. Our first motions demonstrate the connection shared between these nominally distinct bundles. 


\begin{definition}
Given topological quivers $\Q = (E^0, E^1, r, s, \lambda), \mathcal{R} = (F^0, F^1, r, s, \mu)$, a \emph{quiver morphism} $\phi:\Q\to \mathcal{R}$ is a pair of continuous maps $\phi = (\phi^0, \phi^1)$, where $\phi^0:E^0\to F^0$ and $\phi^1:E^1\to F^1$ such that $r\circ\phi^1 = \phi^0\circ r$ and $s\circ\phi^1 = \phi^1\circ s$. 

An \emph{isomorphism of topological quivers} is a quiver morphism with each map $(\phi^0, \phi^1)$ a homeomorphism, and which preserves the measurements of the $s$-system. By this, we mean that for every vertex $v$ and every measurable set $E\subseteq s^{-1}(v)$, $\lambda_v(E) = \mu_{\phi^0(v)}(\phi^1(E))$. It follows that the inverse pair $\phi^{-1} = ((\phi^0)^{-1},(\phi^1)^{-1}) $ is also a quiver morphism. The group of all automorphisms of a topological quiver will be denoted by $\text{Aut}(\Q)$. 
\end{definition}

A locally compact group acts on a topological quiver $\Q$ provided there is a continuous group homomorphism $\alpha: G\to \text{Aut}\Q$, with $g\mapsto \alpha_g = (\alpha_g^0, \alpha_g^1)$. The vertex action $\alpha^0_g(v)$ will often be abbreviated by $g\cdot v$, and likewise for the edge action. The action is \emph{free} provided it is free on the vertices; it follows that the action is free on the edges. The action is \emph{proper} provided the translation map is proper on the vertex space. This is also enough to deduce properness of the action on the edge space by \cite{amenable}*{Proposition 2.1.14}.

\begin{ex}
Given a topological quiver $\Q$, group $G$ and a cocycle $\kappa$, the skew product quiver $\Q\times_\kappa G$ comes equipped with an action of $G$ by right translation on the second factor, $(v, h)\cdot g = (v, hg)$. Indeed, 
\begin{align*}
 s(e, h)\cdot g &= (s(e), h)\cdot g = (s(e), hg) = s(e, hg) = s((e,h)\cdot g),\\
 r(e, h)\cdot g &= (r(e),\kappa(e)h)\cdot g = (r(e), \kappa(e)hg) = r(e, hg) = r((e, h)\cdot g),
\end{align*}
demonstrating covariance of the action with the range and source maps. Continuity is clear, and since right translation on a group is a free and proper action, it follows that the this action of $G$ on $\Q$ is both free and proper. 
\end{ex}

Given a free and proper action of a group on a topological quiver, it is easy to identify the quotient by this action. 
\begin{prop}\label{oq}
 If a group $G$ acts freely and properly on a topological quiver $\Q= (E^0, E^1, r, s, \lambda)$, then the ensemble $q(\Q)= (q(E^0), q(E^1), \dot{r}, \dot{s}, \dot{\lambda})$ is a topological quiver.  
\end{prop}

\begin{proof}
Since $G$ acts freely and properly on the vertex and edge spaces the orbit spaces are locally compact Hausdorff, and the orbit maps $q:E^i\to q(E^i)$, assigning the points to their orbits are quotient maps. The covariance of the range and source maps for the actions on the vertex and edge spaces guarantees unique continuous maps $\dot{r}$ and $\dot{s}$ given by composition with $q$, and $\dot{s}$ remains open since compositions of open maps are open. Since the $s$-system is equivariant by assumption, there is also an associated $\dot{s}$-system $\dot{\lambda}$ by Theorem \ref{orbitsystems}.
\end{proof}

\begin{cor}\label{skeworbit}
Given a skew product $\Q\sk G$, the quotient quiver $q(\Q\sk G)$ is isomorphic to $\Q$.
\end{cor}

\begin{proof}
Since any skew product quiver $\Q\times_\kappa G$ is naturally equipped with a free and proper action, it follows from Proposition \ref{oq} that the quotient is another topological quiver. In this context the orbit maps $q$ take the particularly simple form of projection onto the first factor of the products $E^i\times G$, the quotient maps $\dot{r}, \dot{s}$ coincide with the range and source maps of the original quiver, and the equivariant system of measures agrees with the system from the original quiver, so that $q(\Q\times_\kappa G)\cong \Q$. 
\end{proof}

So we see that given a quiver $\Q$, a group $G$ and a cocycle $\kappa:E^1\to G$, we can generate a new topological quiver $\Q\times_\kappa G$, and this quiver comes equipped with a free and proper action by $G$. For the case of directed graphs and discrete groups, \cite{GrossTucker}*{Theorem 2.2.2} shows that free actions by $G$ are in one to one correspondence with the skew products by some cocycle (see also \cite{dkq}*{Corollary 3.7}). We have shown that for topological quivers, every skew product construction naturally associates a free and proper action by $G$, so we can ask under what criteria this free and proper action recovers the underlying quiver and cocycle. It turns out that there is only one. 

\begin{thm}\label{classify}
Given a topological quiver $\Q$, a locally compact group $G$, and a cocycle $\kappa: E^1\to G$, the skew product quiver $\Q\times_\kappa G$ comes equipped with a free and proper action of $G$ by right translation in the second coordinate, and the quiver of orbits is isomorphic to the initial quiver $\Q$. 

Conversely, suppose we are given a topological quiver $\mathcal{R}$ together with a free and proper action by a group $G$. If the vertex orbit map $q$ is trivial (so that $F^0\cong q(F^0) \times G$), then $\mathcal{R}$ is isomorphic to a skew product quiver.
\end{thm}

\begin{proof}
The first part is Corollary \ref{skeworbit}. We proceed with the other direction. 

Denote the components of $\mathcal{R}$ by $(F^0, F^1, r, s, \lambda)$, and the quiver of orbits by $q(\mathcal{R}) = (E^0, E^1, \dot{r},\dot{s},\dot{\lambda})$ which is guaranteed by Proposition \ref{oq}. By hypothesis, we have an isomorphism \[\phi: F^0\to E^0\times G,\] and we freely identify a vertex from $F^0$ with its image under $\phi$. Since the source map $s$ is $G$-equivariant, Theorem \ref{palais} produces an isomorphism $l$ from $F^1$ to the pullback $E^1*_{\dot{s}} F^0$. Considering that the assignment $(e, (v,g) )\mapsto (e, g)$ with inverse $(e, g)\mapsto (e, s(e), g)$ constructs a homeomorphism $E^1*_{\dot{s}} G\to E^1\times G$, post-composition of $l$ by this map constructs a homeomorphism 
\[\sigma: F^1\to E^1\times G.\]

We claim that the following diagram commutes
\begin{equation*}
\begin{tikzcd}
F^1\ar[swap, "q"]{dd}\ar["s"]{rrr}\ar["\sigma", swap]{dr} & & &  F^0\ar["q"]{dd}\ar["\phi"]{dl}\\
&E^1\times G\ar["\text{pr}_1"]{dl}\ar["\dot{s}\times \text{id}"]{r}&E^0\times G\ar["\text{pr}_1", swap]{dr}& &\\
E^1\ar["\dot{s}", swap ]{rrr} & &  & E^0.\\
\end{tikzcd}
\end{equation*}

It is clear that the bottom square commutes according to the definition of the central map $\dot{s}\times \id$. Meanwhile, the right triangle commutes essentially by definition, and the left triangle commutes as a consequence of Theorem \ref{palais}. For the top square, we compute that 
\[\phi\circ s\circ\sigma^{-1}(e, g) = \phi\circ s (f) = (s(e), g) = \dot{s}\times\id(e, g),\]
where $f$ is the unique element of $F^1$ associated to the data $(e, (s(e), g) )$ in $E^1*_s F^0$. Precomposition by $\sigma$ in the expression above concludes the commutativity of the diagram.


A similar argument shows that the analogous diagram for range maps also commutes, so that we have a commutative diagram 

\begin{equation*}
\begin{tikzcd}
F^0\ar["q", swap]{dd}\ar["\phi", swap ]{dr} & & & F^1\ar["q", swap]{dd}\ar["r", swap]{lll}\ar["s"]{rrr}\ar["\sigma", swap]{dr}\ar["\rho"]{dl} & & & F^0\ar["q"]{dd}\ar["\phi"]{dl}\\
& E^0\times G\ar["\pr"]{dl} & E^1\times G\ar["\dot{r}\times \id", swap]{l}\ar["\pr", swap]{dr} & & E^1\times G\ar["\dot{s}\times\id"]{r}\ar["\pr"]{dl} & E^0\times G\ar["\pr", swap]{dr} & \\
E^0 & & & E^1\ar["\dot{r}"]{lll}\ar["\dot{s}", swap]{rrr} & & & E^0.\\
\end{tikzcd}
\end{equation*}

This diagram now suggests our candidate for a quiver isomorphism with a skew product. We define $\tilde{r} = (\dot{r}\times\id)\circ \rho\circ \sigma^{-1}$, define $\Q=(E^0\times G, E^1\times G, \tilde{r}, \dot{s}\times\id, \{\lambda_{\phi(v)}\}_{v\in F^0})$, and then construct the quiver isomorphism 
\[(\sigma, \phi): \mathcal{R} \to \Q.\] 

Indeed, the maps $\sigma$ and $\phi$ are homeomorphisms by construction, and it is clear that at least the source maps factor through the isomorphism. As for the range map, we compute 
\[\tilde{r}\circ\sigma = (\dot{r}\times\id)\circ\rho = \phi\circ r,\]
verifying the range maps also factor through the isomorphism. The source map $\dot{s}\times \id = \phi\circ s\circ \sigma^{-1}$ is open as a composition of open maps, and the $\dot{s}\times\id$-system agrees with the original $s$-system up to identification of the vertices under $\phi$. Thus, $(\sigma, \phi)$ is a quiver isomorphism. 

It remains to show that the image of this isomorphism is in fact a skew product. We claim that $\Q\cong q(\mathcal{R})\sk G$ for some cocycle $\kappa$. To find this cocycle, we investigate the edge space automorphism $\rho\circ \sigma\inv$.

Since the inner diamond of the above diagram commutes, we know that $\pr\circ\rho\circ\sigma\inv = \pr$, so $\rho\circ\sigma\inv$ fixes the first factor. 
Since $\rho$ and $\sigma$ are $G$-equivariant, this map depends only on the choice of $e\in E^1$, and we define $\kappa$ according to 
\[
\kappa(e) = \pr_2\circ\rho\circ\sigma\inv(e, 1),
\] 

where we treat $1$ as the neutral element of the group. Now, let us identify $q(\mathcal{R})\times_\kappa G$. As spaces, the edge and vertex spaces are $E^1\times G$ and $E^0\times G$, and the source map for the skew product is  $(e, g)\mapsto (\dot{s}(e), g) = \dot{s}\times\id(e, g)$. For the range map, we write 
\[(e, g)\mapsto (\dot{r}(e), \kappa(e)g),\]
and compute 
\begin{align*}
(\dot{r}(e), \kappa(e)g)	&=g(\dot{r}(e), \kappa(e))\\
							&=g[\dot{r}\times\id(e, \kappa(e))]\\
							&=g[\dot{r}\times\id(\rho\circ\sigma\inv(e, 1))]\\
							&=g[\tilde{r}(e, 1)]\\
							&=\tilde{r}(e, g),\\
\end{align*} 
and verify that the range map for the skew product agrees with the range map defined for $\Q$. Finally, the equivariance of the $s$-system $\lambda$ under the group action guarantees that the orbit quiver has the $\dot{s}$-system $\dot{\lambda}$, and the skew product quiver by construction reproduces an equivariant $\dot{s}\times\id$-system $\tilde{\lambda}$ which agrees with the original $s$-system by the commutativity of the top-right square in the diagram. We thus conclude that 
\begin{align*}
\mathcal{R} &= (F^0, F^1, s, r, \lambda)\\
			&\cong (E^0\times G, E^1\times G, \dot{s}\times\id, \tilde{r}, \tilde{\lambda})\\
			&= \Q\\
			&= q(\mathcal{R})\times_\kappa G,
\end{align*}
and $\mathcal{R}$ is isomorphic to a skew-product quiver. 

\end{proof}
\section{Coactions and Skew Products}\label{algebra}



Throughout this section we fix a topological quiver $\Q$ together with a cocycle $\kappa:E^1\to G$, and write $(X, A)$ for the correspondence associated to $\Q$. Furthermore, we write $(Y, B)$ for the correspondence associated to $\Q\times_\kappa G$. 

Dual to the skew product construction for topological quivers, a cocycle also determines an action by pointwise multiplication of $C_0(G)$ on  $(X,A)$ according to the induced homomorphism $\kappa^*:C_0(G)\to M(C_0(E^1) )$. It is clear that this action commutes with the action coming from the range map, since their image is in the representation of $C_0(E^1)$ in $\adj(X)$. 

 Since $\kappa^*$ is a homomorphism of $C_0(G)$, we also have a homomorphism 
 \[
 \overline{\kappa^*\otimes \text{id}}:M(C_0(G)\otimes\cs (G) )\to M(A'\otimes\cs (G) )\subset \adj(X\otimes\cs (G) ),
 \]
since $\kappa^*$ is nondegenerate (that is, $\overline{\text{span}}(\kappa^*(C_0(G))\cdot A') = A'$). Define $v\in M(A'\otimes \cs (G) )$ a unitary by $v=\overline{\kappa^*\otimes\id}(w_G)$. 

 \begin{remark}
Because of the identification $M(C_0(X)\otimes\cs (G) )\cong C_b(X,M^\beta (\cs (G)) )$, we recognize $v$ as the function $v:E^1\to M^\beta(\cs (G) )$ with $v(e)= \iota(\kappa(e)) $. So $v$ provides a \cs-algebraic expression of the cocycle on the associated correspondence. 
 \end{remark}

 \begin{prop}\label{coaction}
Given a cocycle $\kappa$ of $G$ on $\Q$, there is an $(X, A)$ correspondence coaction $\sigma$ by $G$ on $X$ defined by $\sigma(\xi) =v\cdot(\xi\otimes 1)$. 

Moreover, $\sigma$ induces a coaction $\zeta$ by $G$ on $\cs(\Q)$ which is determined on generators by 
\begin{align*}
\zeta\circ k_X (\xi)&= \overline{k_X\otimes\id}\circ\sigma (\xi)\\
\zeta\circ k_A (\xi)&= k_A\otimes\id\\
\end{align*}
where $(k_X, k_A)$ is the universal representation of $(X, A)$ in $\cs(\Q)$. 
 \end{prop}

\begin{proof}
 Because $X$ is nondegenerate as an $A-A$ correspondence, the results of \cite{kqrcoact} apply. In particular, since $\mu=M\circ\kappa^*$ commutes with $\phi_A$,  we apply \cite[Corollary 3.5]{kqrcoact} to generate $\zeta$. 
\end{proof}


With this coaction in hand, we may generate the coaction crossed product $\cs(\Q)\times_\zeta G$ following the theory introduced in the Preliminaries. Our main purpose is to show that this crossed product agrees with the \cs-algebra associated to the skew product quiver. 

\begin{thm}\label{isomorphism}
$\cs(\Q)\times_\zeta G \cong \cs(\Q\times_\kappa G)$.
\end{thm}

\begin{remark}
Here is our strategy to prove this theorem. We will first produce a covariant representation $(\Pi,\psi): (\cs(\Q),\zeta)\to M(\cs(\Q\sk G)$, thereby generating a unique homomorphism 
\[\Pi\times\psi:\cs(Q)\times_\zeta G\to \cs(\Q\sk G).\]
Then, we will invoke the Dual Invariant Uniqueness Theorem \cite[Theorem 5.4]{boiler} to show that the integrated form is injective. To achieve this end, our approach will be through a series of lemmas and propositions. As a guide, it may be constructive to consult the following diagram: 

\begin{equation}\label{integratedform}
\begin{tikzcd}
\cs(\Q)\ar["\Pi", swap]{ddr}\ar["j_{\cs(\Q)}"]{r} & M(\cs(\Q)\times_\zeta G\ar["\Pi\times\psi"]{dd}) & \ar["j_{G}", swap]{l}\ar["\psi"]{ddl}C_0(G) \\
\\
 &M(\cs(\Q\times_\kappa G)). & \\
 \end{tikzcd}
\end{equation} 
We begin with the construction of the homomorphism $\Pi$, which will require the most significant resources. 
\end{remark}


\begin{prop}\label{CPcov}
There is a Cuntz-Pimsner covariant homomorphism $(\mu, \nu): (X,A)\to (M(Y), M(B))$. By the functoriality of Cuntz-Pimsner covarience \cite{kqrfunctor}, there is a unique associated homomorphism $\Pi:\cs(\Q)\to M(\cs(\Q\times_\kappa G))$. 
\end{prop}
 

We define $\nu:A\to M(B)$ according to $\nu:f\mapsto f\circ q$, and turn our attention to the map $\mu$. Throughout this discussion, fix $\xi\in C_c(E^1)$ and $f\in C_c(F^0)$. 

\begin{lem}\label{cpct}
 The function $(\xi\circ q) \cdot (f\circ s): F^1\to \mathbb{C}$ has compact support.
\end{lem}

\begin{proof}
Let $K=\supp(\xi)\subseteq E^1$ and $L=\supp(f)\subseteq F^0$. We may identify $F^1$ with $E^1*_s F^0$ by the map $(e,g)\mapsto (e, s(e), g)$ by Theorem \ref{palais}. Then 
\[
(K\times L)\cap (E^1*_s F^0) = \{(e, v, g): e\in K, (v,g)\in L, \text{ and } s(e)=q(v,g) = v\}
\] 
is compact in $E^1*_sF^0\cong F^1$. This is exactly the support of $\xi\circ q \cdot f\circ s$. 
\end{proof}

\begin{lem}\label{bounded}
 The map $\mu(\xi): C_c(F^0)\to C_c(F^1)$ defined by $\mu(\xi)\cdot f = \xi\circ q\cdot f\circ s$ is uniform to Hilbert-norm bounded, and hence extends to a bimodule map $\mu(\xi): B\to Y$. 
\end{lem}

\begin{proof}
Lemma \ref{cpct} shows that $\mu(\xi)$ is well defined, and linearity is clear. So see that 
\begin{align*}
\langle\mu(\xi)\cdot f, \mu(\xi)\cdot f\rangle(v,g) &= \int \overline{(\xi\circ q)\cdot (f\circ s)}\cdot(\xi\circ q)\cdot (f\circ s)\,d\lambda_{v,g}\\
&= \int |\xi\circ q|^2|f\circ s|^2\,d\lambda_{v,g}\\ 
&= \langle\xi, \xi \rangle (q(v,g))|f(v,g)|^2\leq \|\xi\|^2\|f\|_u^2.\\
\end{align*}
Since this holds for all vertices $(v, g)$, we see that
\begin{equation}\label{bdd}
\|\mu(\xi)\cdot f\|\leq \|\xi\|\|f\|_u. 
\end{equation}
So $\mu(\xi)$ is bounded, and thus extends uniquely to a bounded linear map $B\to Y$. 
\end{proof}

\begin{prop}
There is a Hilbert to uniform-norm bounded map $\mu(\xi)^*: C_c(F^1)\to C_c(F^0)$. Moreover, the unique continuous linear extension is an adjoint for $\mu(\xi)$. 
\end{prop}

\begin{proof}
Fix $\eta\in C_c(F^1)$ and define $\mu(\xi)^*$ by 
\[\mu(\xi)^*\cdot \eta (v,g) = \int \overline{\xi\circ q (\alpha, g)}\eta(\alpha, g)\, d\lambda_{v,g}(\alpha, g).\]
Notice that the restriction of $\xi\circ q$ to $s^{-1}(v,g)$ (and we will denote all such restrictions by, e.g. $\xi\circ q |$) agrees with the restriction of $\xi$ to $s^{-1}(v)$ so that $\xi\circ q|\in C_c(s^{-1}(v,g))$, and  $\mu(\xi)^*\cdot \eta (v,g) = \langle \xi\circ q|, \eta|\rangle$, the inner product for $(L^2(F^1), \lambda_{v,g})$. Then 
\begin{align*}
|\mu(\xi)^*\cdot \eta(v,g)|^2 &= |\langle \xi\circ q|, \eta|\rangle|^2\leq \left[\int|\xi\circ q|^2\,d\lambda_{v,g}\right]\cdot \left[\int |\eta|^2 \,d\lambda_{v,g}\right]\\
&=\langle\xi\circ q, \xi\circ q\rangle(v,g)\cdot \langle\eta,\eta\rangle(v,g)\leq \|\xi\|^2\cdot\|\eta\|^2,
\end{align*}
so $\mu(\xi)^*\eta$ is bounded from $C_c(F^1)$ with the Hilbert norm to $C_c(F^0)$ with the uniform norm. So $\mu(\xi)^*$ extends uniquely to $\mu(\xi)^*: Y\to B$. 

Lastly, to check adjoints, for $f\in C_0(F^0)$, 
\begin{align*}
&\langle\mu(\xi)^*\cdot\eta, f\rangle(v,g) =\overline{\int \overline{\xi\circ q}\cdot \eta \, d\lambda_{v,g}}\cdot f(v,g) = \int (\xi\circ q)\cdot \overline{\eta}\,d\lambda_{v,g}\cdot f(v,g)\\
&=\int (\xi\circ q) \cdot \overline{\eta} \cdot f\circ s\,d\lambda_{v,g} = \langle\eta, (\xi\circ q)\cdot (f\circ s)\rangle(v,g) = \langle\eta, \mu(\xi)\cdot f\rangle(v,g). 
\end{align*}
It follows that $\mu(\xi)^*$ is an adjoint for $\mu(\xi)$. 
\end{proof}

\begin{remark}\label{relativemultiplier}
It now follows that $\mu(\xi)\in \adj(B, Y) = M(Y)$ for any $\xi\in C_c(E^1)$. Since the map $\mu:C_c(E^1)\to M(Y)$ is bounded by \eqref{bdd}, $\mu$ extends to a linear map $\mu:X\to M(Y)$. Now, for any $\xi\in X$, we already have a natural right action of $B$ on $\mu(\xi)$ given by evaluation, but we also have a left action of $B$ on $\mu(\xi)$, defined by $f\cdot \mu(\xi) = (f\circ r)\cdot (\xi\circ q)$. Arguments analogous to the proofs of Lemmas \ref{cpct} and  \ref{bounded} show that $f\cdot \mu(\xi)\in Y$, and thus $\mu$ maps $X$ to the relative multiplier algebra $M_B(Y)$.
\end{remark}

\begin{prop}
The pair $(\mu,\nu):(X,A)\to (M(Y), M(B))$ is a correspondence homomorphism.
\end{prop}

\begin{proof}
It suffices to confirm that $\langle \mu(\upsilon), \mu(\xi)\rangle = \nu(\langle\upsilon, \xi\rangle)$, and that $\mu(f\cdot \xi) = \nu(f)\cdot\mu(\xi)$. We first check covariance of the left action, by computing 
\begin{align*}
\mu(f\cdot \xi) &= (f\circ r\cdot \xi)\circ q = (f\circ r)\circ q )\cdot (\xi\circ q)\\
 &= (f\circ q)\circ r\cdot (\xi\circ q) = \nu(f)\cdot \mu(\xi).
\end{align*} 

To confirm the relationships between the inner products, we first choose $f\in B$ and compute that for any vertex $(v,g)$, 
\begin{align*}
\langle \mu(\eta), \mu(\xi)\rangle\cdot f(v,g) &=\mu(\eta)^*\mu(\xi)\cdot f (v,g) = \mu(\eta)^*[\xi\circ q\cdot f\circ s](v,g)\\
&= \int \overline{\eta\circ q}\cdot (\xi\circ q)\cdot(f\circ s) \,d\lambda_{v,g} = \int\overline{\eta\circ q}\cdot (\xi\circ q)\,d\lambda_{v,g} f(v,g)\\
&=\int (\overline{\eta}\cdot\xi)\circ q\,d\lambda_{v,g} f(v,g) = \left[\int \overline{\eta}\cdot \xi\, d\lambda_v\right]\circ q \cdot f(v,g)\\
&= \langle\eta, \xi\rangle \circ q \cdot f(v,g) = \nu(\langle\eta, \xi\rangle)\cdot f(v,g).
\end{align*} 
Since $f$ and $(v,g)$ were arbitrary, the result follows. 
\end{proof}

With this correspondence homomorphism now in hand, we now turn our attention to Cuntz-Pimsner covariance for this homomorphism. 

\begin{lem}\label{kat mult}
We have $\nu(J_X)\subseteq M(B; J_Y)$. 
\end{lem}

\begin{proof}

For topological quivers, the ideal $J_X$ is characterized in the preliminaries as $J_X=C_0(E^0_{\text{rg}})$, the set of regular vertices of the topological quiver. 
It suffices to show that for any $f\in J_X = C_0(E^0_\rg)$ and $b\in B$, we have $\nu(f)\cdot b \in J_Y$. So suppose $\nu(f)\cdot b(v,g)\neq 0$, and we will show that $(v,g)$ is a regular vertex of $\Q\times_\kappa G$. It follows that  $f\circ q (v,g) = f(v)\neq 0$, so $v\in E^0_\rg$. Since $v$ is not a sink, we can easily find an edge emitted by $(v,g)$. Meanwhile we may choose a suitable compact neighborhood $W\subseteq E^0$ of $v$ with $r^{-1}(W)\subseteq E^1$ compact and $s$ restricted to this subset a local homeomorphism. 

Define a compact subset $U=W\times V$ of $E^0\times G$. We claim that $U$ is a precompact neighborhood of $(v, g)$ which has $r^{-1}(\overline{U})$ compact and $s|_{r^{-1}(U)}$ a local homeomorphism. U is closed, so $r^{-1}(\overline{U})=r^{-1}(U)$ is closed in $E^1\times G$. Moreover, if $(e, g)\in r^{-1}(U)$, then $r(e, g) = (r(e), k(e)g)$, and it follows that 
\[
r^{-1}(U)\subseteq r^{-1}(W)\times \kappa(r^{-1}(W))^{-1}\cdot V.
\]
Note that since the second factor in the above product is compact, and thus $r^{-1}(U)$ is compact. So it remains to show that the restriction of the source map to $r^{-1}(U)$ is a local homeomorphism. This is easy, since the source map acts like the source map for $\Q$ on the first factor and the identity on the second factor, so the restriction is the coordinatewise product of a local homeomorphism and a global homeomorphism.

\end{proof}

\begin{lem}\label{kat commutes}
The diagram 
\begin{equation*}
\begin{tikzcd}
J_X\ar["\phi_A" swap]{d}\ar["\nu"]{r}&M(B;J_Y)\ar["\bar{\phi}_B|"]{d}\\
\mathbb{K}(X)\ar["\mu^{(1)}", swap ]{r}&M_B(\mathbb{K}(Y)),\\
\end{tikzcd}
\end{equation*}
 commutes. 
\end{lem}

\begin{proof}
Choose $f\in J_X$, and fix $\epsilon>0$. Then there is a finite-rank operator $\Theta=\sum \theta_{\eta_k, \xi_k}\in \bK(X)$ with $\|\phi_A(f)-\Theta\|<\epsilon$. Now, for any $y\in C_c(F^1)$, we compute that 
\begin{align*}
&\left\|[\mu^{(1)}\circ \phi_A(f)]y-[\bar{\phi}_B|\circ\nu]y\right\|\\
&\leq\left\|[\mu^{(1)}\circ\phi_A(f)]y-[\mu^{(1)}(\Theta)]y\right\|+\left\|\mu^{(1)}(\Theta)y-[\bar{\phi}_B|\circ \nu]y\right\|\\
&\leq \epsilon\|y\| + \left\|\mu^{(1)}(\Theta)y-[\bar{\phi}_B|\circ \nu]y\right\|,
\end{align*}
so that what remains is to bound the second addend. To this end, we see 
\[\mu^{(1)}(\Theta)y = \mu^{(1)}(\sum\theta_{\eta_k, \xi_k})y = \left[\sum\mu(\eta_k)\mu(\xi_k)^*\right]y,\]
and by evaluating pointwise, that 
\begin{align*}
\mu^{(1)}(\Theta)y(e,g) &= \sum \eta_k\circ q(e, g)\cdot \int \overline{\xi_k\circ q}(\alpha,g)y(\alpha, g)\,d\lambda_{s(e), g}\\
&=\sum \eta_k(e)\cdot \int \overline{\xi_k}(\alpha) y(\alpha, g)\,d\lambda_{s(e)} = \left[\sum \eta_k\langle\xi_k, y(\cdot, g)\rangle\right](e)\\
&=\Theta(y(\cdot, g))(e).
\end{align*}
On the other hand, 
\begin{align*}
[\bar{\phi}_B|\circ\nu](f)y(e,g) &= [\bar{\phi}_B| (f\circ q) y](e,g) = (f\circ q \circ r)(e,g)\cdot y(e,g) \\
&= (f\circ r\circ q)(e, g)\cdot y(e,g)  = [\phi_A(f) y(\cdot, g)] (e), \\
\end{align*}
from which it follows that 
\begin{align*}
\mu^{(1)}(\Theta) y(e,g) - (\bar{\phi}_B|\circ \nu)(f) y(e,g) &= \Theta(y(\cdot, g))(e) - [\phi_A(f) y(\cdot, g)](e)\\
& = ([\Theta-\phi_A(f)]y(\cdot, g))(e).
\end{align*}
So we compute 
 \begin{align*}
 &\|\mu^{(1)}(\Theta)-(\bar{\phi}_B|\circ \nu)(f)y\|^2\\
 & = \sup_{F^0}\langle\mu^{(1)}(\Theta)-(\bar{\phi}_B|\circ \nu)(f)y, \mu^{(1)}(\Theta)-(\bar{\phi}_B|\circ \nu)(f)y\rangle(v, g)\\
 & = \sup_{E^0, G}\langle(\Theta-\phi_A)(f)y(\cdot, g), (\Theta-\phi_A)(f)y(\cdot, g)\rangle(v)\\
 & = \sup_G  \|(\Theta-\phi_A)(f)y(\cdot, g)\|^2\\
 & \leq \sup_G \|(\Theta-\phi_A)(f)\|^2\|y(\cdot, g)\|^2\\
\end{align*}
Taking square roots,

\begin{align*}
&\left\|[\mu^{(1)}\circ \phi_A(f)]y-[\bar{\phi}_B|\circ\nu]y\right\|\\
&\leq\left\|[\mu^{(1)}\circ\phi_A(f)]y-\mu^{(1)}(\Theta)]y\right\|+\left\|\mu^{(1)}(\Theta)y-[\bar{\phi}_B|\circ \nu]y\right\|
\leq 2\epsilon\|y\|
\end{align*}

\end{proof}

We can now show that $(\mu, \nu)$ is Cuntz-Pimsner covariant. 

\begin{proof}[Proof of Proposition \ref{CPcov}]

$\nu$ is clearly nondegenerate, and Remark \ref{relativemultiplier} shows that $\mu:X\to M_B(Y)$, satisfying the first two conditions. Lemmas \ref{kat mult} and \ref{kat commutes} supply the next two conditions of Definition \ref{CPcovcor}, so that $(\mu,\nu)$ is a Cuntz-Pimsner covariant correspondence homomorphism. 

The existence of $\Pi$ is now a consequence of \cite{kqrfunctor}*{Corollary 3.6}.
\end{proof}


Having completed the construction of $\Pi$ by Proposition \ref{CPcov}, we return to the Diagram \ref{integratedform}, and introduce the homomorphism $\psi$. Since $M(\cs(\Q\times_\kappa G))$ contains a copy of $M(C_0(F^0))$, we define $\psi:C_0(G)\to M(C_0(F^0))$ by $\psi(\gamma)=\gamma\circ q_2$, where $q_2$ is the projection of $F^0$ onto the second factor. Then the composition $k_B\circ \psi:C_0(G)\to M(\cs(\Q\sk G))$ induces a coaction, which we denote by $\delta^\psi$.

\begin{prop}\label{covariant}
The pair $(\Pi, k_B\circ\psi)$ is a covariant representation for the coaction $(\cs(\Q), \zeta)$ in $M(\cs(\Q\times_\kappa G))$, so we have a unique integrated form $\cs(\Q)\rtimes_\zeta G\to M(\cs(Q\times_\kappa G))$, which we will denote by $\Pi\times\psi.$  
\end{prop}

\begin{proof}
Our objective is to show that the diagram 
\begin{equation*}
\begin{tikzcd}
\cs(\Q)\ar["\zeta"]{r}\ar["\Pi", swap]{d} 	&M(\cs(\Q)\otimes\cs (G) )\ar["\overline{\Pi\times\id}"]{d} \\
M(\cs(\Q\times_\kappa G))\ar["\delta^\psi", swap]{r}	&M(\cs(\Q\times_\kappa G)\otimes\cs (G) ) \\
\end{tikzcd}
\end{equation*}
 commutes. Since $\cs(\Q)$ is generated by the images of the canonical correspondence maps $(k_X, k_A)$, it is enough to confirm separately that 
 \begin{align*}
 \overline{\Pi\times \id}\circ \zeta\circ k_A &= \delta^\psi\circ\Pi\circ k_A\\
 \overline{\Pi\times \id}\circ \zeta\circ k_X &= \delta^\psi\circ\Pi\circ k_X.\\
 \end{align*}

 Beginning with the first computation, we witness 
\begin{align*}
\overline{\Pi\times \id}\circ \zeta\circ k_A(a) &= \overline{\Pi\times\id}\circ \overline{k_A\times\id} (a\otimes 1) = \overline{\Pi\circ k_A\otimes\id}(a\otimes 1)\\
&=\overline{k_B\circ\nu\otimes\id}(a\otimes 1) = \overline{k_B\otimes \id} (\nu(a)\otimes 1), 
\end{align*}
whose argument we identify as the operator valued function $\nu(a)\otimes 1\in C_b(F^0, M^\beta(\cs (G) ))$ assigning $(v, g)\mapsto a(v)\iota(1)$.  Attending to this argument $\nu(a)\otimes 1$, we compute pointwise
\begin{align*}
\nu(a)\otimes1(v, g) &= a(v)\iota(1)= a(v)\iota(u_gu_g^*) = 1\otimes u_g\cdot a(v)\otimes 1\cdot 1\otimes u_g^*\\
&= [\overline{\psi\otimes\id}(\omega_G)\cdot a\otimes1\cdot\overline{\psi\otimes\id}(\omega_G)^*](v, g).
\end{align*}
Then, since $\overline{k_B\otimes \id}$ is a homomorphism, follows that $\overline{\Pi\times\id}\circ \zeta\circ k_A = \delta^\psi\circ \Pi\circ k_A$.\\

We follow up with the second computation, which is a bit more delicate. Observe that 
\begin{align*}
&\delta^{\psi}\circ\Pi\circ k_X(x)\\
=&\delta^\psi\circ k_Y\circ\mu(x)\\
=&\overline{k_B\circ\psi\otimes\id}(\omega_G)
\big(k_Y(\mu(x))\otimes1\big)
\overline{k_B\circ\psi\otimes\id}(\omega_G)^*\\
=&\overline{k_B\otimes\id}\circ\overline{\psi\otimes\id}(\omega_G) 
\big(k_Y(\mu(x))\otimes 1\big) 
\overline{k_B\otimes\id}\circ\overline{\psi\otimes\id}(\omega_G)^*\\
=&\overline{k_B\otimes\id}\circ\overline{\psi\otimes\id}(\omega_G)
\big(\overline{k_Y\otimes\id}(\mu(x)\otimes 1)\big) 
\overline{k_B\otimes\id}\circ\overline{\psi\otimes\id}(\omega_G)^*\\
=&\overline{k_Y\otimes\id}\Big[\overline{\psi\otimes\id}(\omega_G)
\big(\mu(x)\otimes 1\big)
\overline{\psi\otimes\id}(\omega_G)^*\Big],
\end{align*}
where the last line follows since $(k_Y, k_B)$ is a Toeplitz representation. Now, isolating the argument (which is a multiplier of $Y\otimes \cs (G)$), we calculate pointwise for the elementary tensor $b\otimes\gamma\in B\otimes \cs (G)$
\begin{align*}
&\overline{\psi\otimes\id}(\omega_G)(\mu(x)\otimes 1)\overline{\psi\otimes\id}(\omega_G)^* (b\otimes\gamma)(e, g)\\
=&\overline{\psi\otimes\id}(\omega_G)(r(e),\kappa(e)g)\cdot x(e)1\cdot\overline{\psi\otimes\id}(\omega_G)^*(s(e), g) b(s(e), g)\gamma\\
=&u_{\kappa(e)g}\cdot x(e)1\cdot u_{g}^*\cdot b(s(e))\gamma\\
=&x(e)u_{\kappa(e)g}u_g^*\cdot b(s(e), g)\gamma\\
=&x(e)u_{\kappa(e)}\cdot b(s(e), g)\gamma\\
=&(v\cdot(x\otimes 1))\circ q(e, g)\cdot b\otimes\gamma(e, g)\\
=&\overline{\mu\otimes\id}(v\cdot(x\otimes 1))(e, g)\cdot b\otimes\gamma(e, g).\\
\end{align*}
Thus, the argument agrees with $\overline{\mu\otimes\id}(v\cdot(x\otimes 1))$ on elementary tensors, and consequently for all of $B\otimes\cs (G)$. 

With this in hand, we observe that 
\begin{align*}
&\delta^\psi\circ\Pi\circ k_X(x) \\
=& \overline{k_Y\otimes\id}\left[\overline{\mu\otimes\id}(v\cdot(x\otimes 1))\right] \\
=& \overline{k_Y\circ\mu\otimes\id}(\sigma(x))\\
=&\overline{\Pi\circ k_X\otimes\id}(\sigma(x)) \\
=& \overline{\Pi\otimes\id}\circ\overline{k_X\otimes\id}\circ\sigma(x) \\
=& \overline{\Pi\otimes\id}\circ\zeta\circ k_X(x)\\
\end{align*}
and we thus conclude that the pair $(\Pi, k_B\circ\psi)$ is a covariant representation of the coaction $(\cs (\Q), \zeta)$.
\end{proof}

It remains for us to show that the integrated form $\Pi\times \psi$ is an isomorphism. We will show that the image of $\Pi\times\psi$ contains all of the generators of $\cs(\Q\times_\kappa G)$ to deduce surjectivity, and employ the Dual Invariant Uniqueness Theorem \cite[Theorem 5.4]{boiler} to prove injectivity. 

\begin{prop}\label{surjective}
The map $\Pi\times \psi: \cs(\Q)\times_\zeta G\to \cs(\Q\times_\kappa G)$ is surjective. 
\end{prop}


\begin{proof}
As stated above, we need only show that $\Pi\times\psi$ is surjective on the images of the generators $k_B(C_c(F^0)), k_Y(C_c(F^1))$. We check for $a\in A$ and $f\in C_0(G)$ that
\begin{align*}
&\quad\Pi\times\psi(j_{\cs\Q}(k_A(a))j_G(f))\\
&= \Pi(k_A(a))k_B(\psi(f))\\
&= k_B(\nu(a))k_B(\psi(f)) = k_B(\nu(a)\psi(f))\\
&= k_B(a\otimes1\cdot 1\otimes f) = k_B(a\otimes f).
\end{align*}
 A similar computation for $x\in C_c(E^1)$ shows 
\begin{align*}
&\quad\Pi\times\psi(j_\Q(k_X(x))j_G(f))\\
&=\Pi(k_X(x))k_B(\psi(f))\\
&=k_Y(\mu(x))k_B(\psi(f))\\
&=k_Y(x\otimes1\cdot 1\otimes f)\\
&=k_Y(x\otimes f),
\end{align*}
Taking these equations together, we see that products such as 
\[k_B(a\otimes f)k_Y(x\otimes \phi)\]
are contained in the image of $\Pi\times\psi.$ An argument in linearity and density now shows that the integrated form is surjective. 
\end{proof}

\begin{prop}\label{injective}
$\Pi\times\psi$ is injective.
\end{prop}

\begin{proof}
We apply \cite[Theorem 5.4]{boiler}. First, this entails the covariant representation $(\Pi, \psi)$ in $M(\cs(\Q\times_\kappa G))$ with $\text{ker }\Pi=\text{ker }j_\Q$. Indeed, the gauge action on the skew product quiver intertwines the gauge action on $\cs(\Q)$, so since $\nu$ is injective, it follows that $\Pi$ is injective.  Then, since $\Pi=(\Pi\times\psi)\circ j_\Q$, it also follows that $j_\Q$ is also injective, so that the kernels agree. 

Next, we need an action $\beta$ of $G$ on $\cs(\Q\times_\kappa G)$ whose strict extension fixes $\Pi$ and acts by right translation on $\psi$. Define $\beta_h$ as the action determined on $\cs(\Q\times_\kappa G)$ by right translation on $(Y, B)$. Then $\overline{\beta_h}(\Pi) = \Pi$, while $\overline{\beta_h}(\psi(f)) =\psi(\text{rt}_hf) $, satisfying the hypotheses of the theorem. It then follows that $\Pi\times\psi$ is injective.  
 \end{proof}

In the course of proving Proposition \ref{injective}, we deduced an auxilary result. 

 \begin{cor}\label{normal}
The coaction $\zeta$ is normal. 
 \end{cor}

 \begin{proof}
In the proof above, we noticed that $j_Q$ is injective, which is the definition of normality of the coaction. 
 \end{proof}

We are now prepared to prove the main theorem. 

 \begin{proof}[Proof of Theorem \ref{isomorphism}]
Proposition \ref{covariant} populates the diagram \ref{integratedform}, while Propositions \ref{injective} and \ref{surjective} demonstrate that the integrated form $\Pi\times\psi$ produces the desired isomorphism. 
 \end{proof}

\section{Application to Duality}\label{amenable}

Topological quivers are quite abundant, as the following proposition shows. 
\begin{prop}
Suppose that $X, Y$ are two second countable locally compact Hausdorff spaces, and suppose that $f, g:X\to Y$ are both continuous maps. If $f$ is open and surjective, then there is a topological quiver with vertex space $Y$ and edge space $X$. 
\end{prop}

\begin{proof}
This is an immediate consequence of Blanchard's Theorem \cite{blanchard}. Since $f$ is an open and surjective map, there is an associated $f$-system of probability measures $\lambda=\{\lambda_y\}_{y\in Y}$. Then the ensemble $(Y, X, g, f, \lambda)$ is the desired topological quiver. 
\end{proof}

In particular, topological quivers generalize directed graphs (where the spaces above are discrete), as well as topological graphs where the source map is a local homeomorphism. Theorem \ref{isomorphism} thus recovers \cite{kqr:graph}*{Theorem 2.4} as well as \cite{tgcoact}*{Theorem 3.1}. We mention that the approach used in this present article constructs the isomorphism $\cs\Q\times_\zeta G\cong \cs(\Q\sk G)$ with our bare hands, a departure from the technique of \cite{tgcoact} which employs Landstad duality to recover the coaction, then shows this coaction is equivalent to the constructed one. Presumably, this technique would work just as well. 

We have shown that the coaction $\zeta$ associated to a cocycle on a topological quiver is injective, and consequently a normal coaction. As mentioned in the preliminaries, it is possible for a coaction to be simultaneously normal and maximal, and Proposition \ref{amenableaction} shows that this has strong implications for the dual action. So far, we have been unable to demonstrate the maximality of $\zeta$, so we offer the claim here as a conjecture. 

\begin{conj}
    If $\Q$ is a topological quiver equipped with a cocycle $\kappa:E^1\to G$, over a locally compact group, then the natural coaction $\zeta$ described in Proposition \ref{coaction} is maximal. Consequently, the regular representation in Proposition \ref{amenableaction} is an isomorphism.
\end{conj}

The conjecture holds trivially for amenable groups, and the analysis in \cite{kqr:graph} shows that this holds for discrete groups and directed graphs without restriction. A positive answer this question would generalize the work in the discrete setting and strengthen the results contained in \cite{tgcoact, dkq}. As a final application, we deduce consequences for the \cs-algebra of a quiver which has a suitable principal action by a locally compact group.

\begin{cor}
If $\Q$ is a topological quiver equipped with a free and proper action $\alpha$ by $G$ for which the vertex orbit bundle $(E^0, q, E^0/G)$ is trivial, then 
\[
\cs(\Q)\times_{\alpha, r} G\cong \cs(q(\Q))\otimes \bK(L^2G).
\]
Consequently, the reduced crossed product is Morita equivalent to the \cs-algebra of the quotient quiver.
\end{cor}

\begin{proof}
Write $\R = q(\Q)$ the quotient quiver under the action. According to Theorem \ref{classify}, $\Q$ is a skew product quiver $\R\times_\kappa G$, so $\cs(\Q) = \cs(\R\times_\kappa G)$ is $G$-equivariantly isomorphic to $\cs(\R)\times_\zeta G$ for the natural coaction $\zeta$. Since $\zeta$ is a normal coaction (see Corollary \ref{normal}) Katayama duality applies \cite{enchilada}*{Theorem A.69}, and 
\[\cs(\Q)\times_{\alpha, r} G \cong \cs(\R)\times_\zeta G \times_\alpha G = \cs(\R)\otimes \bK.\]  
\end{proof}




\end{document}